\documentclass[conference]{IEEEtran}
\usepackage{fullpage,graphicx,psfrag,url,setspace}
\usepackage{cite}
\usepackage{epsf}
\usepackage{color}
\usepackage{slashbox}
\usepackage[small,bf]{caption}

\usepackage{amsmath,amsthm,verbatim,amssymb,amsfonts,amscd, graphicx,algorithmic}
\usepackage{hyperref}

\usepackage{lipsum}
\usepackage{algorithm, algorithmic}

\theoremstyle{plain}

\newtheorem{prop}{Proposition}[section]

\include{def}

\hypersetup{colorlinks=true}

\title{Dynamic change-point detection \\using similarity networks}

\author{\IEEEauthorblockN{Shanshan Cao}
\IEEEauthorblockA{Georgia Tech\\School of Industrial and Systems Engineering\\
Email: css1@gatech.edu
} \and
 \IEEEauthorblockN{Yao
Xie}
\IEEEauthorblockA{Georgia Tech\\
School of Industrial and Systems Engineering\\
Email: yao.xie@isye.gatech.edu}
}

\begin{document}
\maketitle

\begin{abstract}
From a sequence of similarity networks, with edges representing certain similarity measures between nodes, we are interested in detecting a change-point which changes the statistical property of the networks. After the change, a subset of anomalous nodes which compares dissimilarly with the normal nodes. We study a simple sequential change detection procedure based on node-wise average similarity measures, and study its theoretical property. Simulation and real-data examples demonstrate such a simply stopping procedure has reasonably good performance. We further discuss the faulty sensor isolation  (estimating anomalous nodes) using community detection.
\end{abstract}

\section{Introduction}

Sensors are widely used to measure physical quantities, such as temperature, pressure, seismic waves, etc. A critical issue is to monitor the status of sensors in real time for any anomalous or malfunctioning ones and detect them as quickly as possible. In various scenarios, it may be difficult to identify broken sensor is broken by directly examining the observations from individual sensors, for instance, the random observations from individual sensor are dynamic, non-stationary, and constantly changing mean and variance. However, the pairwise comparison between nodes may be stationary and used for change-point detection. In particular, one may exploit the fact that the anomalous sensors generate observations that are dissimilar to normal sensors to identify anomalous ones effectively.

A real-data example is illustrated in Fig. \ref{figure:ave_graph}. The data corresponds to hourly temperature measurements at 42 transformers in August 2015, at a converting station in Shandong Province, China. Note that the temperature measurements at each sensor have different means and the means are changing dynamically over time. There are six anomalous sensors (shown in the right panel of the figure), whose observations compare very differently from the observations of the normal sensors (shown in the left panel of the figure) at each moment.

\begin{figure}[!htbp]\label{obs_plot}
  \centering
    \includegraphics[width = .8\linewidth]{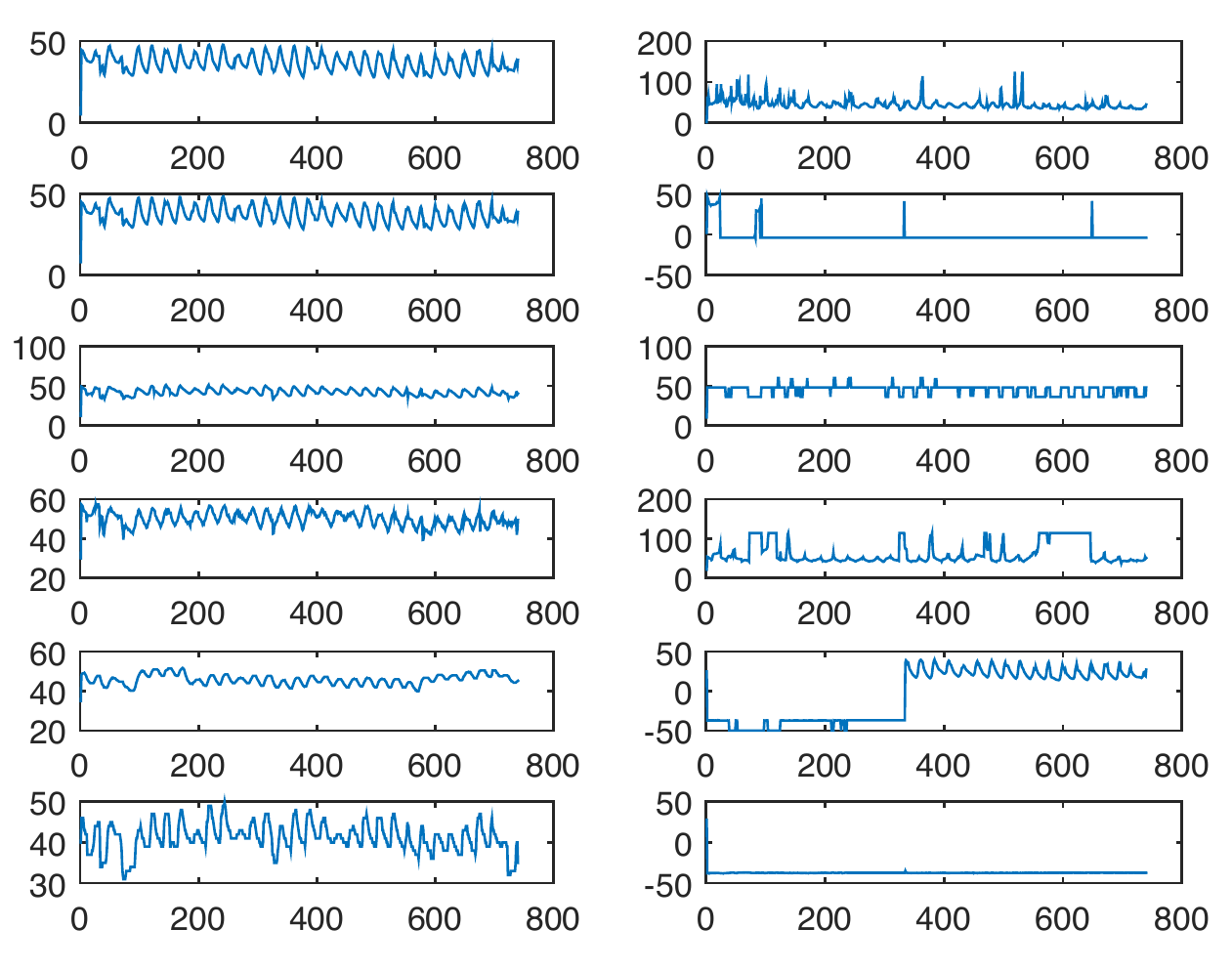}
     \caption{Real data: Hourly temperature measured at 42 transformers in August 2015, at a converting station in Shandong Province, China. The right panel shows the observations from the abnormal sensors while the left panels show the observations from the normal sensors. The instantaneous observations for normal sensors compare similarly to each other; whereas the abnormal sensors compare dissimilarly to the normal sensors.}
     \label{figure:ave_graph}
\end{figure}

In this paper, we are aiming to detect a change from high-dimensional streaming data, by performing the pairwise comparison between sensors at each time. This is equivalent to detecting a change from a sequence of the so-called {\it similarity} networks. Before the change, the similarity between nodes are large, and after the change, a subset of nodes (represent anomalous sensors) have small similarity with the rest of the nodes (represent normal sensors).

We study this problem via a sequential hypothesis test framework and present a simple procedure to detect the change quickly online. The procedure computes the average similarity at each node at each time and detects a change whenever the smallest similarity drops below a certain threshold. We present the general performance bound of this problem and characterize the performance of our simple procedure. We demonstrate the good performance of our procedure using simulations and real-data examples. Finally, we also present a faulty sensor localization method by casting the problem as {\it community detection}.

Close related works include multi-channel and multi-sensor change-point detection
\cite{tartakovsky2008asymptotically,mei2010efficient,xie2013sequential}, which detects the change by constructing statistic at each sensor separately, and \cite{chen2015spectral}, which studies the fundamental information-theoretic limits of recovering variables from their pairwise comparison (pairwise difference, in particular).


\section{Similarity networks}\label{sec:formulation}

Suppose there are $N$ sensors. Sensor $i$ generates a sequence of observations $x_{i,t}$, $t = 1, 2, \ldots$, $i = 1, 2, \ldots, N$. Consider observations in a sliding window of length $w$ at each time $t$:
\[
X_{i,t} = \begin{bmatrix}
x_{i, t-w+1} & \cdots & x_{i, t}
\end{bmatrix}^\intercal \in \mathbb{R}^w.
\]
Using these observations, we construct pairwise similarity between pairs of nodes at each time, which is defined as
\[y_{ijt} = f(X_{i, t}, X_{j, t}),\] where $f$ is a similarity measure. Here assume we may not necessarily observe the complete graph, i.e., all pairwise similarity measures. Fig. \ref{fig:form} illustrates the setting.

Various similarity measures $f$ have been used in practice including Euclidean distance, empirical entropy, inner product, and the Pearson's and Spearman's correlation coefficient. The Pearson's correlation is used to measure the angle between two standardized vectors (standardize a vector by subtracting the mean and then divide by the standard deviation) and it is typically used to estimate the linear dependence of two random variables.

The choice of window length $w$ should be large enough so that the sample similarity will be precisely estimated and converges to a steady value. On the other hand, $w$ should not be too large which will lead to a large detection delay. Fig. \ref{figure:windowSize} illustrates the idea: we plot sample correlation versus $w$ between sensors for the thermo-sensor data. Note that the values of $y_{ijt}$ converge roughly when $w$ is greater than 200.

\begin{figure}[!htbp]
  \centering
    \includegraphics[width = .6\linewidth]{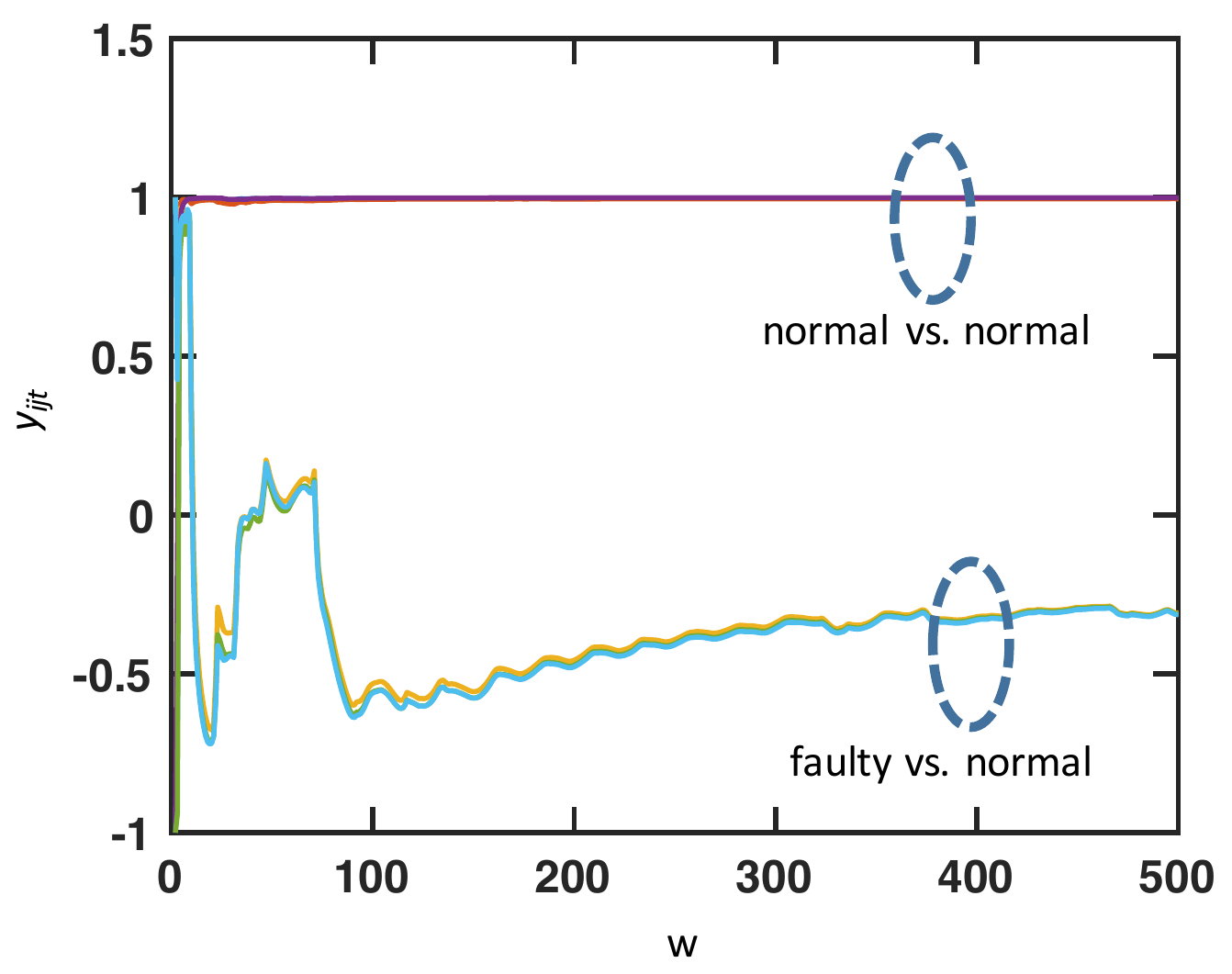}
     \caption{Pairwise Pearson's correlation for the sensor data in Fig. 1, when increasing the window length $w$.}
          \label{figure:windowSize}
\end{figure}

This way we define a sequence of similarity networks; $[y_{ijt}]_{1\leqslant i \leqslant N, 1\leqslant j \leqslant N}$ can be viewed as an adjacency matrix at time $t$. In many applications,  the data streams $x_{i, t}$ themselves are not stationary, but their similarity measures are stationary. Our goal is to detect a change occurs to the similarity networks which changes the distribution of the $y_{ijt}$. We are particularly interested in detecting the emergence of an anomalous community, i.e., after the change, there is a subset of nodes which compare {\it dissimilarly} with the normal nodes.

 \begin{figure}
\begin{center}
\begin{tabular}{cc}
\includegraphics[width = .45\linewidth]{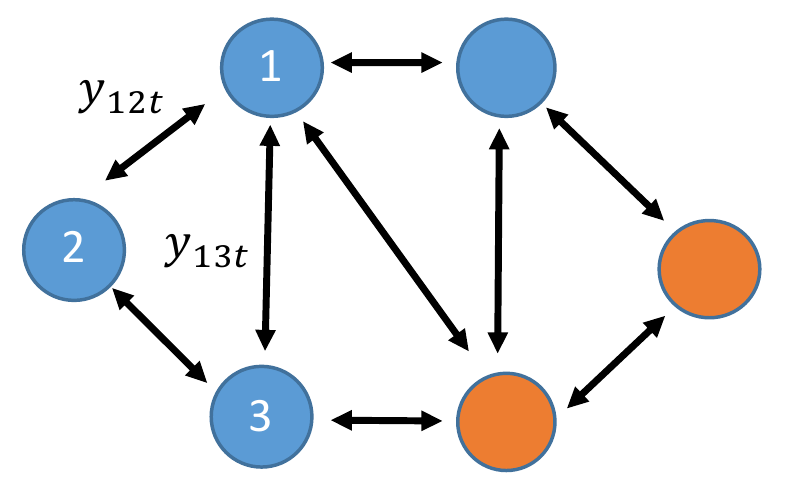} &
 \includegraphics[width = .45\linewidth]{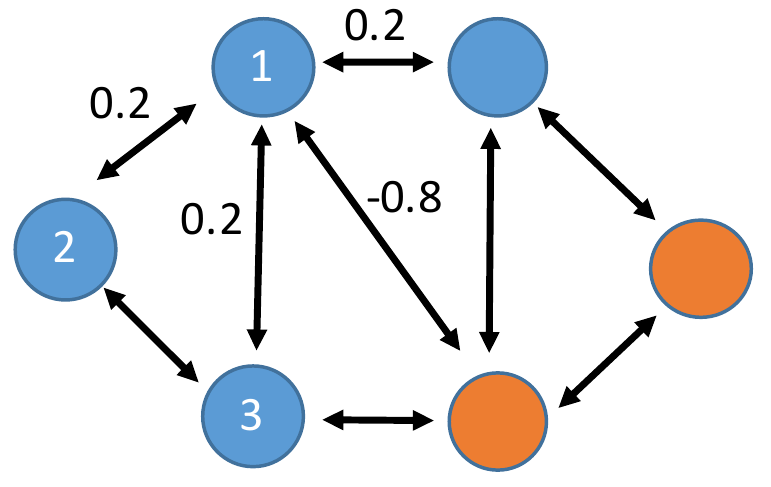}\\
 (a) & (b)
 \end{tabular}
\end{center}
\caption{(a): Similarity network, with the weights on the edges represent similarity between nodes; red nodes represent anomalous sensors which compares dissimilar to the blue (normal) sensors; 
(b): A counter example where node-wise average of correlation coefficient may not lead to correct identification of abnormal nodes.}
\label{fig:form}
\vspace{-0.1in}
\end{figure}

\section{Change-point detection over a sequence of similarity networks}\label{sec:similaritynetwork}

In this section, we present the change-point detection problem as a sequential hypothesis testing problem and present an stopping procedure, which detects the change based on the node-wise average of similarity measures.

Under the null hypothesis,  there is no change, all pairs of the sequences compares similarly to each other on the observations over time. Under the alternative hypothesis, there exists a change-point $\kappa$, $0 \leq \kappa < \infty$ and a subset $\mathcal{S}\subseteq [ 1,N ]$, such that before $\kappa$, the observations from each sensor compares similarly to each other, while after the change, the observations for sensors from $\mathcal{S}$ compares  dissimilarly with normal sensors. Our goal is to detect the change-point as quickly as possible after it occurs and, after we have detected the change,  localize the subset $\mathcal{S}$ of abnormal sensors.
Formally, this can be stated as the following sequential hypothesis test
\begin{equation*}
\begin{array}{ll}
H_0: & y_{ijt} \overset{\text{i.i.d.}}{\sim} \mathbb{P}_0, \quad t = 1, 2, \ldots\\
H_1: &y_{ijt} \overset{\text{i.i.d.}}{\sim} \mathbb{P}_0, \quad t = 1, 2, \ldots, \kappa\\
&   y_{ijt} \overset{\text{i.i.d.}}{\sim} \mathbb{P}_0, \quad i, j \notin \mathcal{S},\\
& y_{ijt} \overset{\text{i.i.d.}}{\sim} \mathbb{P}_1, \quad i \in \mathcal{S}, j \notin \mathcal{S},\\
& ~~~~~~~~~~~~~~~ \mbox{~or~} i \notin \mathcal{S}, j \in \mathcal{S},  \\
& y_{ijt} \overset{\text{i.i.d.}}{\sim} \mathbb{P}_2, \quad i,j \in \mathcal{S}, \quad t = \kappa +1, \ldots, t.
\end{array}
\end{equation*}
Here the unknown $\mathbb{P}_0$ stochastically dominates unknown $\mathbb{P}_1$. The distribution of $y_{ijt}$ for all normal sensors {\it stochastically dominate} \cite{casella2002statistical} those $y_{ijt}$ for $i \in \mathcal{S}$, which captures the idea that the normal sensors compare similar with each other, and the abnormal sensors compare dissimilar with the normal sensors. The unknown $\mathbb{P}_2$, which corresponds to the distribution of comparing faulty sensor with faulty sensors, is not utilized for our test since it usually does not necessarily contain useful information.

To detect emergence of a change, we consider a detection statistic, which is based on node-wide average of similarity measures. For node $i$, let $\mathcal{N}_i$ denote its neighborhood, and define the {\it negative} average similarity over the neighborhood as
\begin{equation}\label{sta_rho}
\rho_{it} =  - \frac{\sum_{j\in \mathcal{N}(i)} y_{ijt}}{|\mathcal{N}(i)|}.
\end{equation}
At each time $t$, we detect a change whenever the maximum of $\rho_{it}$ over all sensors $i = 1, 2, \ldots, N$ exceeds certain threshold. The detection procedure is a stopping time
\begin{equation}\label{stop}
T = \inf\{t:  \max_{i=1}^N  \rho_{it}
> b\},
\end{equation}
where $b$ is a pre-specified threshold to control the false alarm rate.

\section{Theoretical analysis} \label{sec:theo}

\subsection{General bound relating ARL and EDD}

The two standard performance metrics are the Average Run Length (ARL), denoted as $\mathbb{E}^\infty[T]$, which is the expected value of the stopping time when there is no change, and the expected detection delay (EDD), \[\mathbb{E}^1[T] = \sup_{\kappa \geq 1}(\mbox{ess} \sup \mathbb{E}_\kappa [(T-\kappa + 1)^+|\{y_{ij\kappa - 1}\}, \forall i, j ]).\] We have the following general lower bound, which is obtained from the standard result in \cite{lorden71} by calculating the Kullback-Leibler divergence:
\begin{prop}
For $\mathbb{E}^\infty[T]\geq \gamma$, as $\gamma \rightarrow \infty$,
\[\mathbb{E}^1[T] \leq  \frac{\log \gamma}
{\textsf{cut}(\mathcal{S})\cdot \mathsf{KL}\left(\mathbb{P}_1\left\{ y_{ijt}\right\} \text{ }\|\text{ }\mathbb{P}_0\left\{ y_{ijt}\right\} \right)} + O(1),\]
where $\textsf{cut}(\mathcal{S})$ is number of edges cut by separating the true subset $\mathcal{S}$, and $\mathsf{KL}$ denotes the Kullback-Leibler divergence between two distributions.
\end{prop}

\subsection{Correlation networks and performance measure}

In the following, we characterize the performance of our simple detection procedure based on node-wise average similarity (since we did not assume known $\mathbb{P}_0$ and $\mathbb{P}_1$, otherwise we may be able to use likelihood ratio statistic). We show that when the
Pearson's correlation coefficient is used, then zero is a separating threshold.

The Pearson's correlation score for each pair of sensors at time $t$ is defined as
\[
y_{ijt} = \frac{(X_{i, t} - \bar{X}_{i, t})^\intercal (X_{j, t} - \bar{X}_{j, t})}
{\|X_{j, t} - \bar{X}_{j, t}\|\cdot \|X_{j, t} - \bar{X}_{j, t}\|},\]
where $\bar{X}_{i, t} = \textbf{1}_w (\sum_{\ell = t-w+1}^t x_{i, \ell})$, which is a scaled all-one vector.
Note that in this case, $y_{i,j,t} \in [-1, 1]$. If we define $u_{i, t} = (X_{i, t} - \bar{X}_{i, t})/\|X_{i, t} - \bar{X}_{i, t}\|$, then essentially
$y_{ijt}= u_{i, t}^\intercal u_{j, t}$ is the inner-product of two random unit length vectors formed by observations at node $i$ and node $j$.

When the window-length $w$ is large, we may be able to argue (e.g., using the central limit theorem), that the distribution of $y_{ijt}$ are approximately normal with mean $u_{i, t}^\intercal u_{j, t}$ and certain variance $\sigma^2$. Here $u_{i, t}$ denotes a node variable that represent the mean value of the random inner product vector at time $t$, and they satisfy $\textbf{1}_w^\intercal u_{i, t} = 0$, $\forall i$.

Under the assumption that $y_{ijt}$ are i.i.d for all $i, j$, we will have the following proposition.
\begin{prop}[0 is a separating threshold.] \label{zero_thresh}
Under the null hypothesis, let
\[
\textsf{SNR}_{\rm max} = \max_{i\in [N]} \frac{(\sum_{j\in \mathcal{N}(i)}u_{i, t}^\intercal u_{j, t})^2}{|\mathcal{N}(i)|\sigma^2},
\]
If we set 0 as the threshold in the stopping rule, then at time $t$, the probability of false detection at $t$ is given by \[\mathbb{P}^{\infty}\left\{\max_{i=1}^N  \frac{\sum_{j\in \mathcal{N}(i)} (-y_{ijt})}{|\mathcal{N}(i)|} > 0\right\} \lesssim N e^{-\textsf{SNR}_{\rm max}}.\] Under the alternative hypothesis, i.e., there exists a change, let
\[
\textsf{SNR}_{\rm min} = \max_{i\in \mathcal{S}} \frac{(\sum_{j\in \mathcal{N}(i)}u_{i, t}^\intercal u_{j, t})^2}{|\mathcal{N}(i)|\sigma^2},
\]
assume that $\frac{\sum_{j\in \mathcal{N}(i)}u_{i, t}^\intercal u_{j, t}}{\sqrt{|\mathcal{N}(i)|\sigma^2}} < 0 $ for all $i \in \mathcal{S}$, the probability of detection is given by
\[\mathbb{P}^{\kappa}\left\{\max_{i=1}^N  \frac{\sum_{j\in \mathcal{N}(i)} (-y_{ijt})}{|\mathcal{N}(i)|} > 0\right\} \gtrsim 1-e^{-\textsf{SNR}_{\rm min}}.\]
\end{prop}
It can be seen from Proposition \ref{zero_thresh} that the false detection and the detection powers are related to the node-wise inner products $u_{i, t}^\intercal u_{j, t}$. The conditions therein can be interpreted as requirements on the average coherence of the frame formed by these unit length vectors $[u_{1, t}, \ldots, u_{N, t}]$ \cite{mixon2011frame}, as illustrated in Fig. \ref{fig:coherence}.

Note that the inverse of the false detection probability under the null distribution can be approximated to be the ARL. Under the alternative, the probability of detection is related to the EDD of the procedure.

 \begin{figure}
\begin{center}
\includegraphics[width = .3\linewidth]{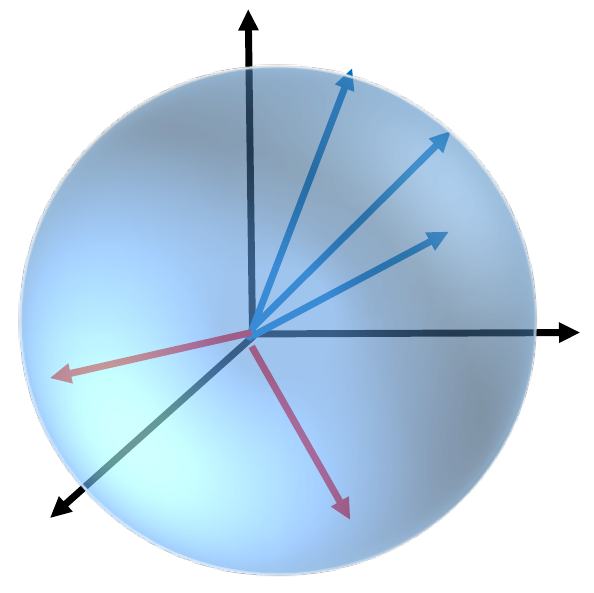}
\end{center}
\caption{Coherence among a group of unit vectors. The blue vectors denote the $u_{i, t}$ at the normal sensors, and the red vectors denote the $u_{i, t}$ associated with the abnormal nodes. The unit-length vectors associated with the normal nodes can be rotating over time, but they are always aligned in a narrow cone; the unit-length length vectors associated with the abnormal nodes are in the ``opposite direction'' from the normal nodes.}
\label{fig:coherence}
\vspace{-0.1in}
\end{figure}

\section{Numerical examples}\label{sec:sim}

In the following, we present numerical and real-data examples to demonstrate the performance of our procedure defined in (\ref{stop}).

\subsection{Simulations}

\subsubsection{Normal observations with linearly increasing or decreasing mean} We first consider a setting where the sensor observations are normal random variables with linearly increasing or decreasing means. Before the change, the trend of the means for all sensors are positive; after the change, a subset of abnormal sensors have the trend become negative, while the normal sensors remain to have a positive trend. Specifically, under the null, we assume sensor observation mean is $t$ and the variance is 25. While under the alternative, 5 of the sensors become abnormal and start to produce observations with a negative trend. To study the performance of our procedure at different SNR levels, we perform experiments with slopes after change being $-.1, -.2, \cdots, -1$.

Under the null, we simulated 5000 observations for each of the 40 sensors. At each time, we computed the pairwise Pearson's correlation using the previous $w =25$ observations. We choose the threshold $b$ by simulation so that ARL is approximately 5000. While under the alternative, for each of the slope, we simulate the first 24 observations under the null and let the change happen at time 25. Fig. \ref{figure:EDD25} shows the detection delays as SNR changes in the alternative, which demonstrates that the procedure can detect change fairly quickly (note that in this case, the noise variance is large so EDD is relatively large).

\begin{figure}[!htbp]
  \centering
    \includegraphics[width = .5\linewidth]{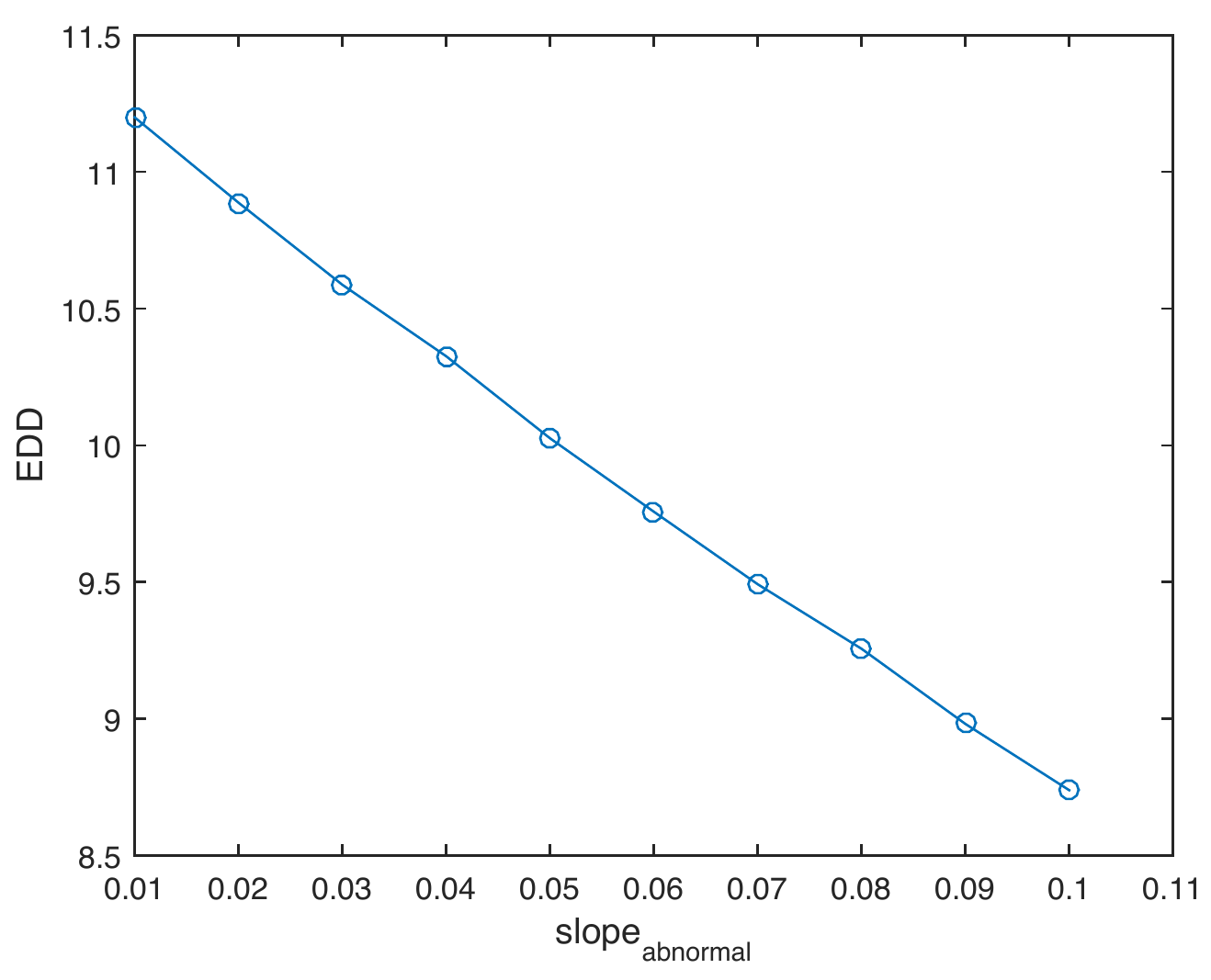}
     \caption{Normal observation with time-varying mean, EDD of our procedure $T$ defined in (\ref{stop}) versus the magnitude of change, which is proportional to the slope of the abnormal nodes.
     }
     \label{figure:EDD25}
\end{figure}

\subsubsection{Normal observations with constant mean and change happens in the covariance}
In this setting, we assume the sensors procedure normal observations with constant mean, however, the change alters the covariance between the sensors. Assume there are $N = 40$ sensors. Under the null, the observations have covariance $\mathbb{E}[x_{i,t} x_{j,t}] \geq 0.5$. Under the alternative, five sensors are anomalous and their correlation with the normal sensors are negative. We simulate 1000 observations. Fig. \ref{figure:compound} demonstrates the histogram of $y_{i, j, t}$, for $i, j = 1, \ldots, 40$, $t = 1, \ldots, 1000$ under the null and the alternative. Note that under the alternative, zero is clearly a separating threshold in this case.

\begin{figure}[!htbp]
  \centering
    \includegraphics[width = .6\linewidth]{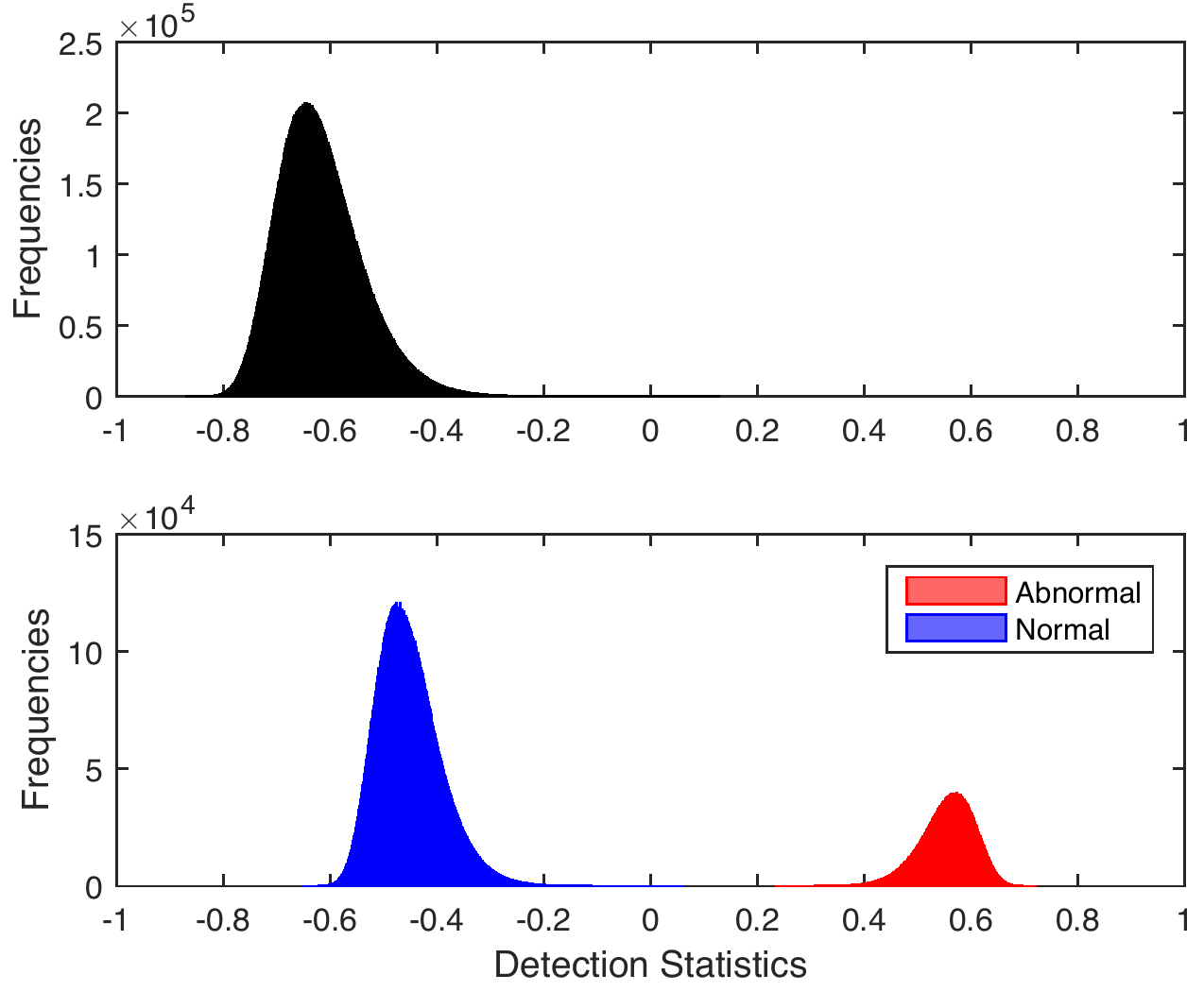}
     \caption{Normal observation with covariance change: histogram of $-y_{ijt}$ under the null and under the alternative. Note that under the alternative, the distribution of normal and abnormal sensor observations split into two groups.
     }
     \label{figure:compound}
\end{figure}

\subsection{Real-data} \label{sec:real}
In this example, we apply our method to the temperature of transformer data illustrated in Fig. \ref{figure:ave_graph}. There are $N = 42$ sensors. Each sensor produced an observation hourly. We have the data for all the August (31 days) with few missing values. We remove the rows with missing values for all the 42 sensors and then we have 739 observations from each sensor. Since they are temperature measurement for transforms in a same local region in a city, the normal sensor readings are expected to have high ``similarity'' at each time. There are six abnormal sensors among the total 42 sensors, and the remaining 36 sensors are normal.  Fig. \ref{figure:heatmap}(a) demonstrates the histogram of $y_{ijt}$ for all sensors across all times, where one can see that the similarities between the normal sensors statistically dominate those between the abnormal and the normal sensors. Therefore, the real data satisfy our assumption and our procedure can be used to detect the change-point. In this example, we find that threshold $b = 0.4721$ can separate the normal and abnormal sensors accurately and hence detect the change immediately.

\begin{figure}[!htbp]
\begin{tabular}{cc}
     \includegraphics[width = .45\linewidth]{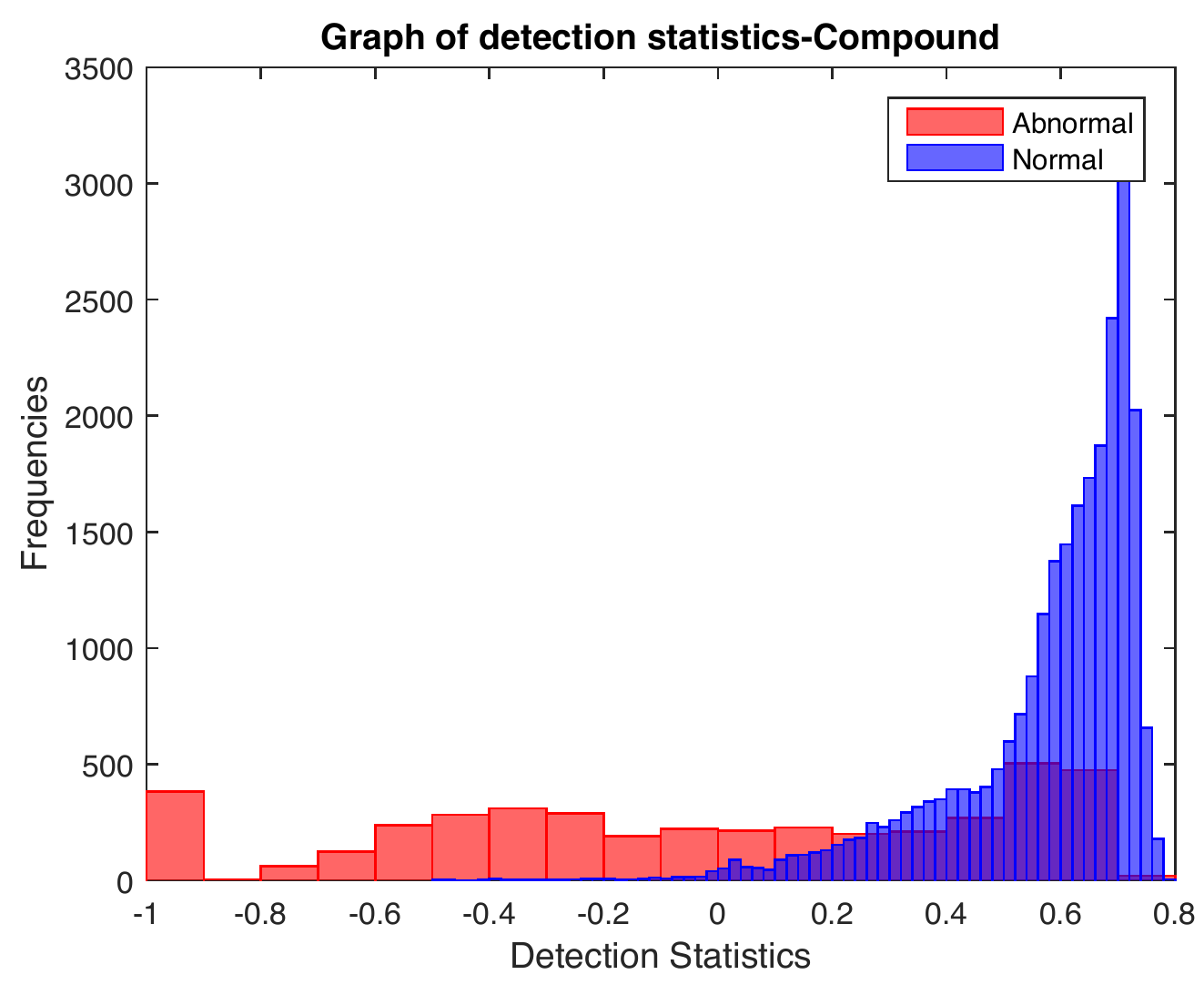} &
         \includegraphics[width = .44\linewidth]{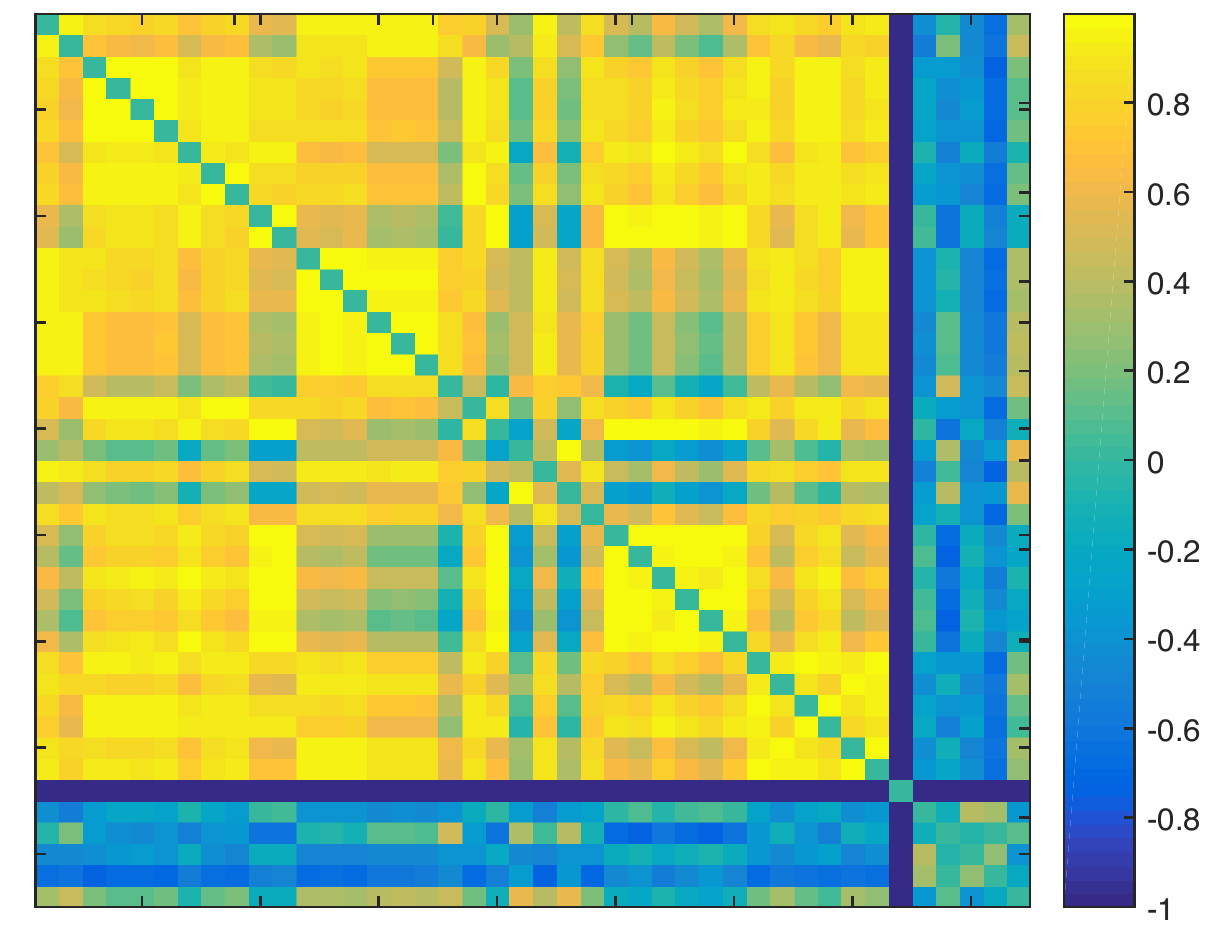} \\
     (a) & (b)
     \end{tabular}
     \caption{(a)  $y_{ijt}$ for a set of normal and abnormal sensors; it is clear that the statistics for the abnormal sensors are significantly larger than that of the normal sensors; (b) Heat map of the correlation at the time when a change is declared.}
     \label{figure:heatmap}
\end{figure}

\section{Fault isolation}\label{sec:isolation}

After we have detected a change,  we can estimate set of abnormal sensors as sensors with corresponding $\rho_{it}$ exceeding $b$, i.e., detect the node $i^*$ to be anomalous when
\[
i^*: \sum_{j\neq i^*} y_{i^*jt} > b'
\]
for a pre-determined threshold $b'$. Although the node-wise average similarity works well for change-point detection, this naive strategy may not work well for fault isolation. For instance, Fig. \ref{fig:form}(b) is a counterexample where the naive method fails to isolate the faulty sensor. The reason is that only local information is used and we make a decision for each sensor individually.

Instead, we consider community detection for fault isolation, which considers the membership assignment of nodes jointly. Since the similarity between the normal nodes tends to be larger than those between the abnormal and the normal nodes, we may estimate the faulty sensors by solving the following optimization problem:
\begin{equation}\label{weight}
   \max_{x_{i} \in \{1, -1\}} x^\intercal Y_T x,
\end{equation}
where $Y_T \in \mathbb{R}^{N\times N}$ is the observed adjacency matrix at the time of detection $T$. The solution has the meaning of membership, where $x_{i} = 1$ indicate $i \in \mathcal{S}$, and  otherwise $i \notin \mathcal{S}$. This corresponds to the general community detection problem (see, e.g., \cite{Abbe16,Vershynin16}), which can be solved via convex relaxation of $x_{i} \in \{1, -1\}$ to $x_{i} \in \mathcal{R}$, and various performance guarantees and efficient algorithms exist.

Fig. \ref{figure:heatmap}(b) shows the heat map of the correlation matrix for the transformer data at $t = 25$, arranged using the recovered membership vector. Clearly, there are two obvious communities and this approach can recover the group of anomalous sensors.

\section{Conclusion and Discussion}\label{sec:con}
We have considered detecting change-point in a sequence of similarity networks, such that the change alters the similarity between sensors. We show that a simple node-wise average similarity procedure can be used to detect the change, and also show that the anomalous sensors can be localized using a community detection.

There are several  modifications of our algorithm to achieve better performance in specific cases. For instance, if we can use to estimate the distributions $\mathbb{P}_0$ and $\mathbb{P}_1$, then we may construct likelihood ratios and apply the classical CUSUM procedure \cite{basseville1993detection} for change detection. For example, the similarity $y_{ijt}$ for the generator temperature data in Fig (\ref{figure:heatmap})(a) can be fit using two Beta distributions.

\section*{Acknowledgement}

The authors work are partially support by NSF grants CCF-1442635 and CMMI-1538746. The authors would like to thank helpful discussions with Dr. Yuxin Chen at Stanford University.

\bibliography{pairwise}

\appendix

\begin{proof}[Proof of Proposition \ref{zero_thresh}]
Under the null, $u_{i, t}^\intercal u_{j, t} > 0$
\begin{equation*}
\begin{split}
&\mathbb{P}^{\infty}\{\max_{i=1}^N  \frac{\sum_{j\in \mathcal{N}(i)} (-y_{ijt})}{|\mathcal{N}(i)|} > 0\}\\
\leq &~N\max_{i = 1}^N\mathbb{P}^{\infty}\{ \{\frac{\sum_{j\in \mathcal{N}(i)} (-y_{ijt})}{|\mathcal{N}(i)|} > 0\}\}\\
\leq &~N\max_{i = 1}^N(1 - \Phi(\frac{\sum_{j\in \mathcal{N}(i)}{\mu}_{it}^T{\mu}_{jt}}{\sqrt{|\mathcal{N}(i)|\sigma^2}}))\\
\leq &~N(1 - \Phi(\sqrt{\textsf{SNR}_{\rm max}})
\lesssim N e^{-\textsf{SNR}_{\rm max}}.
\end{split}
\end{equation*}
Under the alternative, $u_{i, t}^\intercal u_{j, t} < 0$ if $i \in \mathcal{S}, j \notin \mathcal{S}$ or $j \in \mathcal{S}, i \notin \mathcal{S}$.
\begin{equation*}
\begin{split}
&\mathbb{P}^{\infty}\{\max_{i=1}^N  \frac{\sum_{j\in \mathcal{N}(i)} (-y_{ijt})}{|\mathcal{N}(i)|} > 0\}\\
\geq &~\max_{i = 1}^N\mathbb{P}^{\infty}\{ \{\frac{\sum_{j\in \mathcal{N}(i)} (-y_{ijt})}{|\mathcal{N}(i)|} > 0\}\}\\
\geq &~\max_{i = 1}^N(1 - \Phi(\frac{\sum_{j\in \mathcal{N}(i)}{\mu}_{it}^T{\mu}_{jt}}{\sqrt{|\mathcal{N}(i)|\sigma^2}}))\\
\geq &~(1 - \Phi(-\sqrt{\textsf{SNR}_{\rm min}})
\gtrsim 1-e^{-\textsf{SNR}_{\rm min}}.
\end{split}
\end{equation*}
\end{proof}

\end{document}